\documentclass[12pt]{amsart}
\makeatletter
\usepackage{amsfonts,delarray,amssymb,amsmath,amsthm,a4,a4wide}

\usepackage{color}

\title[Blow up of subcritical quantities of the mean curvature flow]{Blow up of subcritical quantities at the first singular time of the mean curvature flow}
\author{Nam Q.  Le}
\address{Department of
Mathematics, Columbia University, New York,
 USA}
\email{namle@math.columbia.edu}

\newcommand{\review}[2][\right]{\relax
\ifx#1\right\relax \left.\fi#2#1\rvert}
\let\abs=\envert
 
\newtheorem{theorem}{Theorem}[section]
\newtheorem{propo}{Proposition}[section]
\newtheorem{remark}{Remark}[section]

\newtheorem{lemma}{Lemma}[section]

\newcommand{\bef}{\begin{flushright}}
\newcommand{\eef}{\end{flushright}}

\newcommand{\eval}[2][\right]{\relax
\ifx#1\right\relax \left.\fi#2#1\rvert}

\let\abs=\envert

\numberwithin{equation}{section}
\let\norm=\enVert
\newcommand\e{\varepsilon}
\newcommand{\h}{\hspace*{.24in}}

\def\h{\hspace*{.24in}}

\def\beq{\begin{eqnarray*}}
\def\eeq{\end{eqnarray*}}

\def\RR{\mbox{$I\hspace{-.06in}R$}}

\begin{document}
\maketitle
\author
\pagenumbering{arabic}
\begin{abstract}
Consider a family of smooth immersions $F(\cdot,t): M^n\to \mathbb{R}^{n+1}$ of closed hypersurfaces in $\mathbb{R}^{n+1}$ moving 
by the mean curvature flow $\frac{\partial F(p,t)}{\partial t} = -H(p,t)\cdot \nu(p,t)$, for $t\in [0,T)$. We show that at the first singular time of 
the mean curvature flow, certain subcritical quantities concerning 
the second fundamental form, for example $\int_{0}^{t} \int_{M_{s}} \frac{\abs{A}^{n + 2}}{ log (2 + \abs{A})} d\mu ds,$ blow up. Our result is a log improvement of recent results of Le-Sesum, Xu-Ye-Zhao where the scaling invariant 
quantities were considered. 
\end{abstract}
\noindent
\section{Introduction}
\h Let $M^{n}$ be a compact $n$-dimensional hypersurface without boundary, and let $F_{0}: M^{n}\rightarrow \RR^{n+ 1}$ be a smooth immersion of $M^{n}$ into $\mathbb{R}^{n+1}$. Consider a smooth one-parameter family of immersions
\begin{equation*}
 F(\cdot, t): M^{n}\rightarrow \RR^{n +1}
\end{equation*}
satisfying
\begin{equation*}
 F(\cdot, 0) = F_{0}(\cdot)
\end{equation*}
and 
\begin{equation}
 \frac{\partial F(p, t)}{\partial t} = -H(p, t)\nu(p,t)~\forall (p, t)\in M\times [0, T).
\label{MCF1}
\end{equation}
Here $H(p, t)$ and $\nu(p, t)$ denote the mean curvature and a choice of unit normal for the hypersurface $M_{t} = F(M^{n},t)$ at $F(p, t)$.
We will sometimes also write $x(p, t) = F(p, t)$ and refer to (\ref{MCF1}) as to the mean curvature flow equation. For any compact $n$-dimensional hypersurface $M^{n}$ which is smoothly embedded in $\RR^{n+1}$ by $F: M^{n}\rightarrow \RR^{n+1}$, 
let us denote by $g = (g_{ij})$ the induced metric, $A = (h_{ij})$ the second fundamental form,  $d\mu =\sqrt{\text{det}~(g_{ij})}~dx$ the volume form, 
$\nabla$ the induced Levi-Civita connection
and $\Delta$ the induced Laplacian.
Then the mean curvature of $M^{n}$ is given by $
 H = g^{ij}h_{ij}. $
We will use the following notation throughout the whole paper,
$$||v||_{L^{p, q}(M\times[0,T))} := (\int_0^T\left(\int_{M_t}|v|^p\, d\mu\right)^{\frac{q}{p}} dt)^{\frac{1}{q}},$$
for a function $v(\cdot,t)$ defined on $M\times [0,T)$.\\
\h Without any special assumptions on $M_0$, the mean curvature flow (\ref{MCF1}) will in general develop singularities  in finite time, characterized by a blow up of the second fundamental form $A(\cdot,t)$.

\begin{theorem}[Huisken \cite{Huisken84}]
Suppose $T < \infty$ is the first singularity time  for a compact mean curvature flow. Then $\sup_{M_t}|A|(\cdot,t) \to \infty$ as $t\to T$.
\label{Aunbound}
\end{theorem}
In Le-Sesum \cite{LS, LS2}, Xu-Ye-Zhao \cite{XYZ1}, it was proved that at the first singularity time of the mean curvature flow, certain 
scaling invariant quantities blow-up. Specifically,

\begin{theorem}
 Suppose $T < \infty$ is the first singularity time  for a compact mean curvature flow. Let $p$ and $q$ be positive numbers satisfying
$\frac{n}{p} + \frac{2}{q} =1.$ Then $\norm{A}_{L^{p, q}(M\times [0, t))}\rightarrow \infty$ as $t\to T$. In particular, for $p=q = n + 2$, one has
$ \int_{0}^{t}\int_{M_{s}}\abs{A}^{n +2} d\mu ds\to
\infty$ as $t\to T.$
\label{scalinginv}
\end{theorem}
The proof in \cite{LS2, XYZ1} used a blow-up argument combined with a compactness property of the mean curvature flow \cite{ChenHe}. The proof
in \cite{LS} used a blow-up argument and Moser iteration. \\  
\h In this paper, we give a logarithmic improvement of the above results by showing that a family of subcritical quantities concerning 
the second fundamental form  blows up at the first singular time of 
the mean curvature flow. Our proof covers a large class of such subcritical quantities including $\int_{0}^{t} \int_{M_{s}} \frac{\abs{A}^{n + 2}}{ log (1 + \abs{A})} 
d\mu ds$. For clarity, we will focus on this quantity. Equivalently, we prove the following
\begin{theorem}
Assume that for the mean curvature flow (\ref{MCF1}), we have
\begin{equation}
 \int_{0}^{T} \int_{M_{t}} \frac{\abs{A}^{n + 2}}{ log (1 + \abs{A})} d\mu dt<\infty.
\end{equation}
\h Then the flow can be extended past time $T$.
\label{Abound}
\end{theorem}

Our result is inspired by a recent log improvement of the Prodi-Serrin criteria for Navier-Stokes equations by Chan-Vasseur \cite{ChV}. The usual 
Prodi-Serrin criterion ensures global regularity of a weak Leray-Hopf solution $u$ of the Navier-Stokes equation in dimension 3 provided
that $\abs{u}^5$ is integrable in space time variables. Chan-Vasseur's result shows that the global regularity holds under the condition
that $\abs{u}^5/log (1 +\abs{u})$ is integrable in space time variables.\\
\h Note that, however, the techniques used in \cite{ChV} and in the present article are different. In \cite{ChV}, the authors used De Giorgi's 
technique while in our paper, we use Moser iteration.
\begin{remark}
 To our knowledge, Theorem \ref{Abound} is the first result in geometric evolutions where the finiteness of a slightly subcritical quantity implies
global existence. It would be interesting to obtain similar results in other settings such as the heat flow of harmonic maps or the Ricci flow. 
\end{remark}

\h Let us comment briefly on ideas of the proof of Theorem \ref{Abound}. 
The key point in the proof of Theorem \ref{Abound} is that for any time $t>0$, the second fundamental form $A(\cdot, t)$ can be bounded in an affine
way by $
 \int_{0}^{t}\int_{M_{s}} \abs{A}^{n +3} d\mu ds. $
More precisely, we have the following
\begin{propo}
 For all $\lambda \in (0, 1]$ there is a constant $c_{\lambda}$ such that for all $T\geq\lambda$
\begin{equation}
 \sup_{x\in M_{T}} \abs{A(x,T)} \leq c_{\lambda} (1 + \int_{0}^{T}\int_{M_{t}} \abs{A}^{n +3} d\mu dt).
\end{equation}
\label{keybound}
\end{propo}
Theorem \ref{Abound} then follows from a Gronwall-type argument on $\sup_{x\in M_{t}} \abs{A}(x,t)$.\\
\h The rest of the paper is organized as follows. In Section \ref{Sobolev}, we establish Sobolev inequalities for the mean curvature flow. 
We will use these inequalities to prove reverse Holder and Harnack inequalities in Section \ref{sec-Harnack}. In Section \ref{boundSF}, we prove Proposition
\ref{keybound}. The proof of Theorem \ref{Abound} will be carried out in the final section, Section \ref{sec-proof}, of the paper.

\section{Sobolev Inequalities for the Mean Curvature Flow}
\label{Sobolev}
In this section, we establish a version of Michael-Simon inequality, Lemma \ref{MSlem}, that allows us to derive a 
Sobolev type inequality, Proposition \ref{MStime}, for the mean curvature flow. 
This Sobolev inequality will be crucial for the reverse Holder and Harnack inequalities in the next section. \\
\h The following lemma consists of a slightly modified Michael-Simon inequality whose proof is based on the original Michael-Simon inequality
\cite{MS} together with the interpolation inequalities. By their inequality
there is a uniform constant $c_{n}$, depending only on $n$, such that for any nonnegative, $C^1$ function 
$f$ on a hypersurface $M\subset \mathbb{R}^{n+1}$, the
following  holds
\begin{equation}
\label{eqn-simon}
(\int_M f^{\frac{n}{n-1}}\,d\mu)^{\frac{n-1}{n}} \le c_{n}\int_{M}(|\nabla f|+ |H|f)\, d\mu.
\end{equation}
\begin{lemma}
Let $M$ be a compact $n$-dimensional hypersurface without boundary, which is smoothly embedded in $\RR^{n + 1}$. Let 
\begin{equation} 
Q = \left\{ \begin{aligned} 
\frac{n}{n-2} &\h \text{if}~ n>2\\
<\infty &\h \text{if} ~ n=2                                                                         
\end{aligned}
\right.
\end{equation}
Then, for all Lipschitz functions $v$ on $M$, we have
\begin{equation*}
 \norm{v}^2_{L^{2Q}(M)}\leq c_{n}\left(\norm{\nabla v}^2_{L^{2}(M)} + \norm{H}^{\frac{2(n + 3)}{3}}_{L^{n + 3}(M)}\norm{v}^2_{L^{2}(M)}\right)
\end{equation*}
where $H$ is the mean curvature of $M$ and $c_{n}$ is a positive constant depending only on $n$. 
\label{MSlem}
\end{lemma}
\begin{remark}
 The exponent $\frac{2}{3}~(<1)$ appearing in $\norm{H}^{n + 3}_{L^{n + 3}(M)}$ in the above inequality plays a crucial role in our paper. It allows us to 
bound $C_{1}$ in terms of $C_{0}$ (defined in (\ref{Czero})). See (\ref{CCbound}).
\end{remark}

\begin{proof} The proof of this lemma is very similar to that of Lemma 3.1 in \cite{LS}. For reader's convenience, we include the proof.
We only need to prove the lemma for $v\geq 0$. Applying Michael-Simon's inequality (\ref{eqn-simon})\cite{MS} to the function $w = v^{\frac{2(n-1)}{n-2}}$, 
we get
\begin{equation*}
 \left(\int_{M}v^{\frac{2n}{n-2}} d\mu\right)^{\frac{n-1}{n}}\leq c_{n}\left(\int_{M}\abs{\nabla v}v^{\frac{n}{n-2}} d\mu + 
\int_{M}\abs{H}v^{\frac{2(n-1)}{n-2}} d\mu\right).
\end{equation*}
By Holder's inequality it follows that
\begin{eqnarray*}
 \left(\int_{M}v^{\frac{2n}{n-2}}d\mu\right)^{\frac{n-2}{n}}&\leq& c_{n}^{\frac{n-2}{n-1}}\left(\int_{M}\abs{\nabla v}v^{\frac{n}{n-2}} d\mu + 
\int_{M}\abs{H}v^{\frac{2(n-1)}{n-2}} d\mu\right)^{\frac{n-2}{n-1}}\\ &\leq &
c_{n}\left (\norm{\nabla v}_{L^{2}(M)}\norm{v}^{\frac{n}{n-2}}_{L^{2Q}(M)} + 
\norm{H}_{L^{n + 3}(M)}\norm{v}^{\frac{2(n-1)}{n-2}}_{L^{2m}(M)} \right)^{\frac{n-2}{n-1}}\\ &\leq &
c_{n} \left (\norm{\nabla v}^{\frac{n-2}{n-1}}_{L^{2}(M)}\norm{v}^{\frac{n}{n-1}}_{L^{2Q}(M)} + 
\norm{H}^{\frac{n-2}{n-1}}_{L^{n + 3}(M)}\norm{v}^{2}_{L^{2m}(M)} \right).
\end{eqnarray*} 
where
\begin{equation*}
 m =\frac{(n-1)(n + 3)}{(n-2)(n + 2)}.
\end{equation*}
Thus
\begin{equation}
  \norm{v}^2_{L^{2Q}(M)}\leq c_{n} \left (\norm{\nabla v}^{\frac{n-2}{n-1}}_{L^{2}(M)}\norm{v}^{\frac{n}{n-1}}_{L^{2Q}(M)} + 
\norm{H}^{\frac{n-2}{n-1}}_{L^{n + 3}(M)}\norm{v}^{2}_{L^{2m}(M)} \right).
\label{firstMM}
\end{equation}
By Young's inequality
\begin{equation}
 ab = (\e^{1/p}a)(\e^{-1/p}b) \leq \frac{\e a^{p}}{p} +\frac{\e^{-q/p}b^{q}}{q}\leq \e a^{p} + \e^{-q/p}b^{q},
\label{Young}
\end{equation}
where $a,b,\e > 0$, $p, q>1$ and $\frac{1}{p} +\frac{1}{q} =1$. If we apply it to (\ref{firstMM}), with 
\begin{equation*}
 a = \norm{v}^{\frac{n}{n-1}}_{L^{2Q}(M)}, \qquad b = \norm{\nabla v}^{\frac{n-2}{n-1}}_{L^{2}(M)},
\end{equation*}
and \begin{equation*}
    \e = \frac{1}{2c_n}, \,\,\, p =\frac{2(n-1)}{n}, \,\,\, q =\frac{2(n-1)}{n-2},
    \end{equation*}
we obtain
\begin{equation*}
  \norm{v}^2_{L^{2Q}(M)}\leq c_{n} \left ( \frac{1}{2c_{n}}\norm{v}^2_{L^{2Q}(M)} + 
(\frac{1}{2c_{n}})^{\frac{-n}{n-2}}\norm{\nabla v}^{2}_{L^{2}(M)} + 
\norm{H}^{\frac{n-2}{n-1}}_{L^{n + 3}(M)}\norm{v}^{2}_{L^{2m}(M)} \right).
\end{equation*}
Hence
\begin{equation}
  \norm{v}^2_{L^{2Q}(M)}\leq c_{n} \left (\norm{\nabla v}^{2}_{L^{2}(M)} + 
\norm{H}^{\frac{n-2}{n-1}}_{L^{n + 3}(M)}\norm{v}^{2}_{L^{2m}(M)} \right).
\label{beforeinter}
\end{equation}
Next, we will use the following interpolation inequality (see inequality (7.10) in \cite{GT})
\begin{equation}
 \norm{u}_{L^{r}}\leq \e  \norm{u}_{L^{s}} + \e^{-\mu}  \norm{u}_{L^{t}}
\label{interpol}
\end{equation}
where $t<r<s$ and 
\begin{equation*}
 \mu = (\frac{1}{t}-\frac{1}{r})/(\frac{1}{r}-\frac{1}{s}). 
\end{equation*}
Note that, in our case
$
 1<m<Q,$
and therefore, by (\ref{interpol})
\begin{equation}
 \norm{v}_{L^{2m}(M)}\leq \e  \norm{v}_{L^{2Q}(M)} + \e^{-\alpha}  \norm{v}_{L^{2}(M)}
\label{interep}
\end{equation}
where $\e>0$ and 
\begin{equation}
 %\alpha = (\frac{1}{2}-\frac{1}{2m})/(\frac{1}{2m}-\frac{1}{2Q}) = \frac{Q(m-1)}{Q-m}.
\alpha = \frac{Q(m-1)}{Q-m} =  \frac{n (2n +1)}{3(n-2)}.
\end{equation}
%We can calculate
%\begin{equation*}
% m-1 = \frac{2n}{(n-2)(n + 1)}
%\end{equation*}
%and
%\begin{equation}
% \alpha = \frac{\frac{2n}{(n-2)(n + 1)}\frac{n}{n-2}}{\frac{2}{(n-2)(n + 1)}} = \frac{n^2}{n-2}.
%\end{equation}
Plugging (\ref{interep}) into the right hand side of (\ref{beforeinter}), we deduce that
\begin{eqnarray}
 \norm{v}^2_{L^{2Q}(M)}&\leq& c_{n}\norm{\nabla v}^{2}_{L^{2}(M)} + c_{n} \norm{H}^{\frac{n-2}{n-1}}_{L^{n + 3}(M)}
\left ( \e  \norm{v}_{L^{2Q}(M)} + \e^{-\alpha}  \norm{v}_{L^{2}(M)} \right)^2\nonumber \\ &\leq &
c_{n}\norm{\nabla v}^{2}_{L^{2}(M)} + c_{n} \norm{H}^{\frac{n-2}{n-1}}_{L^{n + 3}(M)}
\left ( \e^2  \norm{v}^2_{L^{2Q}(M)} + \e^{-2\alpha}  \norm{v}^2_{L^{2}(M)} \right).
\label{absorb1}
\end{eqnarray}
Now, we can absorb the term involving $ \norm{v}^2_{L^{2Q}(M)}$ into the left hand side of (\ref{absorb1}) by choosing
$
 \e^{2} =\frac{1}{2c_{n}}\norm{H}^{-\frac{n-2}{n-1}}_{L^{n + 3}(M)}.$
%Then 
%\begin{equation*}
% \e^{-2\alpha} = (2c_{n})^{\alpha} \norm{H}^{\frac{n-2}{n-1}\alpha}_{L^{n + 2}(M)}
%\end{equation*}
Since $\frac{n-2}{n-1}(1 + \alpha) =  \frac{2(n + 2)}{3}$, we obtain the desired inequality
\begin{equation*}
  \norm{v}^2_{L^{2Q}(M)}\leq c_{n}\norm{\nabla v}^{2}_{L^{2}(M)} +
 c_{n}\norm{H}^{\frac{2(n+3)}{3}}_{L^{n + 3}(M)}\norm{v}^2_{L^{2}(M)}. 
\end{equation*}

%\begin{equation*}
% \frac{n-2}{n-1}(1 + \alpha) = \frac{n-2}{n-1} (1 + \frac{n^2}{n-2}) = n + 2.
%\end{equation*}
\end{proof}
\h Our Sobolev type inequality for the mean curvature flow is stated in the following proposition.

\begin{propo}
For all nonnegative Lipschitz functions $v$, one has
\begin{multline}
 ||v||^{\beta}_{L^{\beta}(M\times[0,T))}\\ \leq c_{n}\max_{0\leq t\leq T}
\norm{v}^{4/n}_{L^{2}(M_{t})}\left (||\nabla v||^2_{L^2(M\times[0,T))} +  \max_{0\leq t\leq T}
\norm{v}^{2}_{L^{2}(M_{t})} ||H||^{\frac{2(n+3)}{3}}_{L^{n+3, \frac{2(n +3)}{3}}(M\times[0,T))}\right),
\end{multline}
where $\beta := \frac{2(n+2)}{n}$.
\label{MStime}
 \end{propo}
\begin{proof}
 By Holder's inequality, we have
\begin{eqnarray*}
 \int_{0}^{T}\int_{M_{t}}v^{\frac{2(n+2)}{n}}d\mu dt = \int_{0}^{T}dt \int_{M_{t}} v^2 v^{4/n}d\mu&\leq&
\int_{0}^{T}dt \left(\int_{M_{t}}v^{\frac{2n}{n-2}}d\mu\right)^{\frac{n-2}{n}}\left(\int_{M_{t}}v^{2}d\mu\right)^{\frac{2}{n}}\\ &\leq &
\max_{0\leq t\leq T}
\norm{v}^{4/n}_{L^{2}(M_{t})} \int_{0}^{T}\norm{v(\cdot, t)}^{2}_{L^{2Q}(M_{t})}.
\end{eqnarray*}
Now, applying Lemma \ref{MSlem}, we get
\begin{multline*}
||v||^{\beta}_{L^{\beta}(M\times[0,T))}\\ \leq c_{n}\max_{0\leq t\leq T}
\norm{v}^{4/n}_{L^{2}(M_{t})}\left (\int_0^T\int_{M_t}|\nabla v|^2\, d\mu\, dt + 
\int_0^T(\int_{M_t}|H|^{n+3}\, d\mu)^{\frac{2}{3}}||v(\cdot,t)||^2_{L^2(M_t)}\,  dt\right)\\
\leq c_{n}\max_{0\leq t\leq T}
\norm{v}^{4/n}_{L^{2}(M_{t})}\left (||\nabla v||^2_{L^2(M\times[0,T))} +  \max_{0\leq t\leq T}
\norm{v}^{2}_{L^{2}(M_{t})} ||H||^{\frac{2(n+3)}{3}}_{L^{n+3, \frac{2(n +3)}{3}}(M\times[0,T))}\right).
\end{multline*}
\end{proof}

\section{Reverse Holder and Harnack Inequalities}
\label{sec-Harnack}
In this section, we state a soft version of reverse Holder
inequality (Lemma \ref{softRH}) and a Harnack inequality (Lemma \ref{Harnackineq}) for parabolic inequality during the mean curvature flow.\\
We start with the differential inequality
\begin{equation}
 (\frac{\partial}{\partial t}-\Delta ) v\leq fv, ~v\geq 0
\label{keyeq}
\end{equation}
where the function $f$ has bounded $L^{q}(M\times [0, T))$-norm with $q > \frac{n+2}{2}$.
Let $\eta(t,x)$ be a smooth function with the property that $\eta(0, x) = 0$ for all $x$. 
\begin{lemma}
Let 
\begin{equation}
 C_{0} \equiv C_{0}(q)= ||f||_{L^q(M\times[0,T))},\qquad  C_{1} = (1 +\norm{H}^{\frac{2(n+3)}{3}}_{L^{n + 3, \frac{2(n +3)}{3}}(M\times [0, T))})^{\frac{n}{n + 2}},
\label{Czero}
\end{equation}
$\beta > 1$ be a fixed number and $q>\frac{n +2}{2}$. 
Then there exists a positive constant $C_{a} = C_{a}(n, q, C_{0}, C_{1})$ such that
 \begin{equation}
 ||\eta^2 v^{\beta}||_{L^{(n+2)/n}(M\times[0,T))} 
\leq C_{a}\Lambda(\beta)^{1 + \nu}||v^{\beta}\left(\eta^2 + \abs{\nabla\eta}^2 + 
2\eta \abs{(\frac{\partial}{\partial t}-\Delta)\eta}\right)||_{L^1(M\times[0,T))},
\label{RH}
\end{equation}
where 
\begin{equation}
 \nu =\frac{n + 2}{ 2q -(n +2)},
\label{nueq}
\end{equation}
and $\Lambda(\beta)$ is a positive constant depending on $\beta$ such that $\Lambda(\beta)\geq 1$ if 
$\beta\geq 2$ (e.g. we can choose $\Lambda(\beta) = 100\beta$).

In fact,  we can choose
\begin{equation}
 C_{a}(n, q, C_{0}, C_{1}) = (2c_{n}C_{0}C_{1})^{1+\nu}.
\label{Ca}
\end{equation}
\label{softRH}
\end{lemma}
This lemma can be proved similarly as in the proof of Lemma 4.1 in \cite{LS}, using the Sobolev type inequality for the mean curvature
flow established in Proposition \ref{MStime}. \\
\h Next, we show that an $L^{\infty}$-norm of $v$ over a smaller set can be bounded by an $L^{\beta}$-norm of $v$ on a bigger set, 
where $\beta \ge 2$. Fix $x_{0}\in \RR^{n +1}$. Consider the following sets in space and time,
\begin{equation*}
 D = \cup_{0\le t\le 1} (B(x_{0}, 1)\cap M_t);\,\,\, D^{'} = \cup_{\frac{1}{12}\le t\leq 1}(B(x_{0}, \frac{1}{2})\cap M_t).
\end{equation*}
\h Then, we have the following Harnack inequality.
\begin{lemma}
Consider the equation (\ref{keyeq}) with $T= 1$. Let us denote by
$ \lambda = \frac{n +2}{n}$, let $q>\frac{n +2}{2}$ and $\beta\geq 2$. Then, there exists a constant $C_{b}= C_{b}(n, q, \beta, C_{0}, C_{1})$ such that
\begin{equation}
 \norm{v}_{L^{\infty}(D^{'})}\leq C_{b}(n, q, \beta, C_{0}, C_{1}) \norm{v}_{L^{\beta}(D)}.
\label{firstM}
\end{equation}
In the above inequalities, $C_{0}$ and $C_{1}$ are defined by (\ref{Czero}).\\
In fact, we can choose
\begin{equation}
\label{eq-Cb}
C_{b}(n, q, \beta, C_{0}, C_{1}) = (4\lambda^{1+ \nu} C_z\beta^{1+\nu})^{\frac{n^2}{\beta}},
\end{equation}
where 
\begin{equation}
C_{z}(n, q, C_{0}, C_{1}) := 4^2\times 100^{1+\nu}c_{n}C_{a}(n, q, C_{0}, C_{1}).
\label{Cz}
\end{equation}
\label{Harnackineq}
\end{lemma}
The proof of this lemma, using Lemma \ref{softRH} and Moser iteration, is similar to that of Lemma 5. 2 in \cite{LS}.

\section{Bounding the second fundamental form}
\label{boundSF}
In this section, we prove Proposition \ref{keybound}. First, we establish the following rescaled version of Proposition \ref{keybound}.
\begin{propo}
There is a universal constant $c_{0}$ depending only on $n$ such that
if 
\begin{equation}
 \int_{0}^{1}\int_{M_{t}} \abs{A}^{n +3} d\mu dt \leq c_{0}
\end{equation}
 then
\begin{equation}
 \sup_{\frac{1}{2}\leq t \leq 1} \sup_{x\in M_{t}}\abs{A(x,t)} \leq 1.
\end{equation}
\label{smallnorm}
\end{propo}
\begin{proof}[Proof of Proposition \ref{smallnorm}]
Using the evolution 
\begin{equation*}
 (\frac{\partial}{\partial t} -\Delta) \abs{A} ^2 = - 2 \abs{\nabla A}^2 + 2\abs{A}^{4}
\end{equation*}
derived in \cite{Huisken84}, we obtain for $v = \abs{A}^2$
\begin{equation*}
 (\frac{\partial}{\partial t}-\Delta ) v\leq fv
\end{equation*}
where $f =2v.$
Our proposition is now an easy consequence of Lemma \ref{Harnackineq} where $q = \frac{n +3}{2}$ and $\beta = n+ 3$. In fact, 
from (\ref{nueq}) and (\ref{Ca}), one has
$C_{a} = c_{n} (C_{0}C_{1})^{n +3}$. From (\ref{Cz}) and (\ref{eq-Cb}), one has
$C_{b} = c_{n} C_{z}^{\frac{n^2}{n +3}} = c_{n} (C_{0}C_{1})^{n ^2}.$ By Holder's inequality
\begin{equation}
 C_{1} \leq (1 + C_{0}^{n +3})^{\frac{n}{n +2}}.
\label{CCbound}
\end{equation}
Now, by (\ref{firstM}), one has
\begin{equation*}
\norm{v}_{L^{\infty}(D^{'})}\leq C_{b}\norm{v}_{L^{\beta}(D)} \leq c_{n} (C_{0} (1 + C_{0}^{n +3})^{\frac{n}{n +2}})^{n^2} \norm{v}_{L^{\beta}(D)}
\leq c_{n} (c_{0} (1 + c_{0}^{n +3})^{\frac{n}{n +2}}) c_{0}^{\frac{1}{ n +3}} \leq 1
\end{equation*}
 if $c_{0}$ is small, universal.
\end{proof}

\begin{proof}[Proof of Proposition \ref{keybound}]
We first consider the special case $\lambda =1$ and $T\geq 1.$ There are two cases.\\
{\bf Case 1.} This is the case when 
\begin{equation*}
 \int_{0}^{T}\int_{M_{t}} \abs{A}^{n +3} d\mu dt \leq c_{0}.
\end{equation*}
In this case, we consider a new one-parameter family of immersions $\tilde{F}$ defined by $\tilde{F}(x,t) = F(x, T-1 + t)$. Then
\begin{equation*}
 \int_{0}^{1}\int_{\tilde{M_{t}}} \abs{\tilde{A}}^{n +3} d\mu dt = \int_{T-1}^{T}\int_{M_{t}} \abs{A}^{n +3} d\mu dt \leq \int_{0}^{T}\int_{M_{t}} \abs{A}^{n +3} d\mu dt
\leq c_{0}.
\end{equation*}
By Proposition \ref{smallnorm}, one has
\begin{equation*}
 \sup_{x\in \tilde{M_{1}}} \abs{\tilde{A}(x,1)} \leq 1. 
\end{equation*}
Hence 
\begin{equation}
 \sup_{x\in M_{T}} \abs{A(x, T)} \leq 1.
\end{equation}
{\bf Case 2.} This is the case when 
\begin{equation*}
 \int_{0}^{T}\int_{M_{t}} \abs{A}^{n +3} d\mu dt \geq c_{0}.
\end{equation*}
In this case, we consider a new one-parameter family of immersions $\tilde{F}$ defined by 
$
 \tilde{F}(x,t) = Q F(x, \frac{t}{Q^2}). $
We find that
\begin{equation*}
 \int_{0}^{Q^2T}\int_{\tilde{M_{t}}} \abs{\tilde{A}}^{n +3} d\mu dt =\frac{1}{Q}\int_{0}^{T}\int_{M_{t}} \abs{A}^{n +3} d\mu dt = c_{0}
\end{equation*}
if we choose
\begin{equation*}
 Q = \frac{1}{c_{0}} \int_{0}^{T}\int_{M_{t}} \abs{A}^{n +3} d\mu dt \geq 1.
\end{equation*}
Now, we are back in {\bf Case 1} and thus can conclude
\begin{equation*}
 \sup_{x\in \tilde{M}_{Q^2 T}} \abs{\tilde{A}}(x, Q^2T) \leq 1.
\end{equation*}
This gives
\begin{equation*}
 \sup_{x\in M_{T}} \abs{A(x, T)} = Q \sup_{x\in \tilde{M}_{Q^2 T}} \abs{\tilde{A}}(x, Q^2T)
\leq Q =  \frac{1}{c_{0}} \int_{0}^{T}\int_{M_{t}} \abs{A}^{n +3} d\mu dt. 
\end{equation*}
Combining the above two cases, we find that for $T\geq 1$, one has
\begin{equation}
 \sup_{x\in M_{T}} \abs{A(x, T)} \leq Q = (1 + \frac{1}{c_{0}})(1 + \int_{0}^{T}\int_{M_{t}} \abs{A}^{n +3} d\mu dt).
\label{special}
\end{equation}
Finally, we consider the general case $\lambda \in (0, 1]$ and $T\geq \lambda$. As usual, let us consider 
a new one-parameter family of immersions $\tilde{F}$ defined by 
$
 \tilde{F}(x,t) = Q F(x, \frac{t}{Q^2}) $
where $Q = \frac{1}{T^{\frac{1}{2}}}\leq \frac{1}{\lambda^{\frac{1}{2}}}.$
Then $Q^2 T\geq 1.$
Thus, from the estimate (\ref{special}) in the special case, one has
\begin{eqnarray*}
  \sup_{x\in \tilde{M}_{Q^2T}} \abs{\tilde{A}}(x, Q^2T) &\leq& (1 + \frac{1}{c_{0}})(1 + \int_{0}^{Q^2T}\int_{M_{t}} \abs{\tilde{A}}^{n +3} d\mu dt)
\\ &=& (1 + \frac{1}{c_{0}})(1 + \frac{1}{Q}\int_{0}^{T}\int_{M_{t}} \abs{A}^{n +3} d\mu dt).
\end{eqnarray*}
Consequently,
\begin{eqnarray*}
 \sup_{x\in M_{T}}\abs{A}(x, T) &=& Q\sup_{x\in \tilde{M}_{Q^2T}} \abs{\tilde{A}}(x, Q^2T) \leq 
Q(1 + \frac{1}{c_{0}})(1 + \frac{1}{Q}\int_{0}^{T}\int_{M_{t}} \abs{A}^{n +3} d\mu dt)\\ &\leq &
\frac{1}{\lambda^{\frac{1}{2}}}(1 + \frac{1}{c_{0}})(1 + \int_{0}^{T}\int_{M_{t}} \abs{A}^{n +3} d\mu dt).
\end{eqnarray*}

\end{proof}
\begin{remark}
 We can choose the constant $c_{\lambda}$ in Proposition \ref{keybound} as follows:
$
 c_{\lambda} = \frac{1}{\lambda^{\frac{1}{2}}} (1 + \frac{1}{c_{0}}).$

\end{remark}
\section{Proof of the main theorem}
\label{sec-proof}
\begin{proof}[Proof of Theorem \ref{Abound}] Fix $\tau_{1}<T$ such that $0< \tau_{1}<1$. Then, by Proposition \ref{keybound},
 for any $t\geq\tau_{1}$, there 
is a universal constant $c$ depending only on $\tau_{1}$, such that
\begin{equation}
  \sup_{x\in M_{t}} \abs{A(x,t)} \leq c (1 + \int_{0}^{t}\int_{M_{s}} \abs{A}^{n +3} d\mu ds).
\label{tauineq}
\end{equation}
Let $f(t) =   \sup_{x\in M_{t}} \abs{A(x,t)}$, $\Psi (s) = s log (2 +s)$ and 
\begin{equation*}
 G(s) = \int_{M_{s}} \frac{\abs{A}^{n + 2}}{log (2 + \abs{A})} d\mu.
\end{equation*}
 Then $\Psi$ is an increasing function. Note that (\ref{tauineq}) gives
\begin{eqnarray*}
 f(t)&\leq& c (1 + \int_{0}^{t} \int_{M_{s}}\Psi(\abs{A}) \frac{\abs{A}^{n + 2}}{log (2 + \abs{A})} d\mu ds) \\
&\leq& c (1 + \int_{0}^{t} \Psi (\sup_{x\in M_{s}} \abs{A(x,s)}) \int_{M_{s}} \frac{\abs{A}^{n + 2}}{log (2 + \abs{A})} d\mu ds)= 
c (1 + \int_{0}^{t} \Psi (f(s)) G(s) ds).
\end{eqnarray*}
Let 
\begin{equation*}
 h(t) = c (1 + \int_{0}^{t} \Psi (f(s)) G(s) ds).
\end{equation*}
Then for $t\geq \tau_{1}$
\begin{equation*}
 f(t)\leq h(t)
\end{equation*}
and
\begin{equation*}
 h^{'}(t) = c \Psi (f(t)) G(t) \leq c \Psi (h(t)) G(t).
\end{equation*}
Let $\tilde{\Psi} (y) = \int_{c}^{y} \frac{1}{\Psi (s)} ds$. Then for $t\geq \tau_{1}$
\begin{equation*}
 \tilde{\Psi} (h(t)) - \tilde{\Psi} (h(\tau_{1})) \leq c \int_{\tau_{1}}^{t} G(s) ds\leq c\int_{0}^{T} G(s) ds <\infty.
\end{equation*}
Hence, since $h(\tau_{1})$ is finite 
\begin{equation*} 
 \sup_{\tau_{1}\leq t< T} \tilde{\Psi}(h(t))\leq \tilde{\Psi} (h(\tau_{1})) + c \int_{0}^{T} G(s) ds<\infty.
\end{equation*}
Since$
 \int_{c}^{\infty} \frac{1}{\Psi (s)} ds = \infty, $
we deduce that $
 \sup_{\tau_{1}\leq t<T} h(t)<\infty. $
Hence
$
  \sup_{\tau_{1}\leq t<T} f(t)<\infty. $
Therefore, the flow can be extended past T.
\end{proof}

{} 

\end{document}